\documentclass[11pt]{article}
\usepackage[utf8]{inputenc}

\usepackage{amsthm}
\usepackage{amssymb}
\usepackage{amsmath}
\usepackage[hidelinks]{hyperref}
\usepackage{color}
\usepackage{ulem}
\usepackage{graphicx}

\usepackage[left=1in,right=1in,top=1in,bottom=1in]{geometry}

\newtheorem{theorem}{Theorem}
\newtheorem{lemma}[theorem]{Lemma}
\newtheorem{corollary}[theorem]{Corollary}
\newtheorem{proposition}[theorem]{Proposition}

    \newtheorem{innercustomthm}{Theorem}
    \newenvironment{restatetheorem}[1]
    {\renewcommand\theinnercustomthm{#1}\innercustomthm}
    {\endinnercustomthm}

\theoremstyle{remark}
\newtheorem{example}[theorem]{Example}
\newtheorem{remark}[theorem]{Remark}
\newtheorem{definition}[theorem]{Definition}

\newcommand{\R}{\mathbb{R}}                       
\newcommand*{\RP}{\mathbb{R}\mathrm{P}}   
\newcommand{\C}{\mathcal{C}}      
\DeclareMathOperator{\conv}{conv} 
\DeclareMathOperator{\proj}{proj} 
\DeclareMathOperator{\red}{red}   
\DeclareMathOperator{\nerve}{nerve}   
\newcommand*{\eqdef}{\stackrel{\mbox{\normalfont\tiny def}}{=}}  

\title{Colorful Words and $d$-Tverberg Complexes}

\author{
Florian Frick
\thanks{Department of Mathematical Sciences, Carnegie Mellon University, Pittsburgh, PA 15213, USA, and Institut f\"ur Mathematik, Freie Universit\"at Berlin, Arnimallee 2, 14195 Berlin, Germany, frick@cmu.edu\@. Supported by NSF grant DMS 1855591, NSF CAREER grant DMS 2042428, and a Sloan Research Fellowship.}
\and 
R. Amzi Jeffs
\thanks{Department of Mathematical Sciences, Carnegie Mellon University, Pittsburgh, PA 15213, USA, amzij@cmu.edu\@. Supported by the National Science Foundation through Award No. 2103206.}}

\date{\today}

\begin{document}
\normalem 
\maketitle

\begin{abstract}
\noindent
    We give a complete combinatorial characterization of weakly $d$-Tverberg complexes. These complexes record which intersection combinatorics of convex hulls necessarily arise in any sufficiently large general position point set in~$\R^d$. This strengthens the concept of $d$-representable complexes, which describe intersection combinatorics that arise in at least one point set. Our characterization allows us to construct for every fixed~$d$ a graph that is not weakly $d'$-Tverberg for any~${d' \le d}$, answering a question of De Loera, Hogan, Oliveros, and Yang.
\end{abstract}

\section{Introduction}

Tverberg's theorem states that any set of $(d+1)(r-1) + 1$ points in $\R^d$ can be partitioned into $r$ parts so that the convex hulls of these parts share a common point.
Over the last five decades Tverberg's theorem has inspired numerous extensions and variations~\cite{barany2018}. Concurrently, intersection patterns of convex sets have been investigated through $d$-representable complexes, whose $k$-dimensional faces correspond to nonempty $(k+1)$-fold intersections among a family of convex sets in $\R^d$~\cite{tancer2013}.

Recall that a simplicial complex is a downward-closed set system. A simplicial complex $\Delta$ consisting of subsets of $[n] \eqdef \{1,2, \dots, n\}$ is called \emph{weakly $d$-Tverberg} if for any sufficiently large point set $P \subseteq \R^d$ in general position there are pairwise disjoint sets $P_1, \dots, P_n \subseteq P$ such that for every~${\sigma \subseteq [n]}$ we have that $\bigcap_{i \in \sigma} \conv(P_i) \ne \emptyset$ if and only if $\sigma \in \Delta$. Tverberg's theorem states that the $(r-1)$-simplex is weakly $d$-Tverberg for every $d\ge 1$, and $(d+1)(r-1)+1$ points in~$\R^d$ are sufficient to guarantee the existence of $P_1, \dots, P_r$ as above. 

De Loera, Hogan, Oliveros, and Yang~\cite{altered_nerves} initiated the study of \emph{d-Tverberg complexes}, 
which are a restriction of weakly $d$-Tverberg complexes in which $P_1, \dots, P_n$ are required to partition all of~$P$.
Every (weakly) $d$-Tverberg complex is $d$-representable, and in~\cite{altered_nerves} it is shown that the converse fails for $d=2$. Here we show that the converse fails with arbitrarily large dimension gap. More precisely, for every fixed~$d\ge1$ we construct graphs (which are $3$-representable, as follows from more general results of Wegner and Perel\textquotesingle man \cite[Section 3.1]{tancer2013}) that are not weakly $d'$-Tverberg for any~${d'\le d}$. This answers a question raised by De~Loera, Hogan, Oliveros, and Yang.

\begin{theorem}\label{thm:graphs}
For every $d \ge 1$, there is a graph that is not (weakly) $d'$-Tverberg for any $d'\le d$. 
\end{theorem}

Along the way to constructing these graphs, we give a complete combinatorial characterization of weakly $d$-Tverberg complexes.
This characterization will involve finding words whose letters come from $[n]$, with certain alternating subwords corresponding to faces.
Oliveros and Torres~\cite{oliveros_torres} were the first to use this approach in the study of $d$-Tverberg complexes, defining \emph{general $d$\nobreakdash-word-representable graphs} and proving that certain families of such graphs---for example, bipartite graphs---are $d$-Tverberg for appropriate values of $d$.
Inspired by this strategy, we define the class of \emph{$d$-colorfully representable complexes} (Definition \ref{def:d-colorfully-representable}) and connect their combinatorics to the geometry of finite point sets in $\R^d$.
While Oliveros and Torres primarily worked constructively, we proceed in the opposite direction: our proof of Theorem~\ref{thm:graphs} is based on combinatorially recognizing obstructions that prevent a graph from being $d$-Tverberg.

Below, a \emph{word} is simply a finite list of symbols from some chosen alphabet.
A \emph{subword} is any word obtained by deleting some (possibly none) of the instances of letters in a word. 

\begin{definition}
Let $\sigma\subseteq [n]$ be a set of size $r\ge 1$, and let $d\ge 1$. A \emph{$d$-colorful word} on alphabet $\sigma$ is a word $W$ of length $(d+1)(r-1) +1$ such that for every $i\in [d+1]$ the restriction of $W$ to the indices $(i-1)(r-1)+1$ through $i(r-1)+1$ inclusive contains every letter from $\sigma$ exactly once.
We call these segments of consecutive indices \emph{blocks} in $W$.
\end{definition}

For example, the word $\underline{124\overline{3}}\overline{42}\underline{\overline{1}34\overline{2}}\overline{134}$ is a 3-colorful word on the alphabet $\{1,2,3,4\}$.
Above, we have under- and overlined the four blocks in which every letter must appear exactly once.
We note that $d$-colorful words also appear in equivalent forms as ``rainbow partitions" in \cite{por} and ``colorful Tverberg types" in \cite{bukh_loh_nivasch}.
The relationship between the combinatorics of these words and the geometry of Tverberg partitions---a topic which both these papers deal with---will allow us to completely characterize weakly $d$-Tverberg complexes.

\begin{definition}\label{def:d-colorfully-representable}
A simplicial complex $\Delta\subseteq 2^{[n]}$ is called \emph{$d$-colorfully representable} if there is a word $W$ on alphabet $[n]$ so that for every $\sigma\subseteq[n]$, we have $\sigma\in \Delta$ if and only if $W$ contains a $d$-colorful subword on alphabet $\sigma$.
\end{definition}

\begin{theorem}\label{thm:weak-iff-colorful}
A simplicial complex is weakly $d$-Tverberg if and only if it is $d$-colorfully representable. 
\end{theorem}

The forward direction of Theorem \ref{thm:weak-iff-colorful} is a consequence of the existence of linearly ordered point sets in $\R^d$ whose minimal Tverberg partitions correspond to $d$-colorful words---for example the ``stretched diagonal" of Bukh, Loh, and Nivasch~\cite{bukh_loh_nivasch}. The converse follows from a recent universality result of P\'or~\cite{por} on Tverberg partitions.

\begin{example}\label{ex:colorfully-representable}
\begin{figure}[h!]
    \centering
    \includegraphics{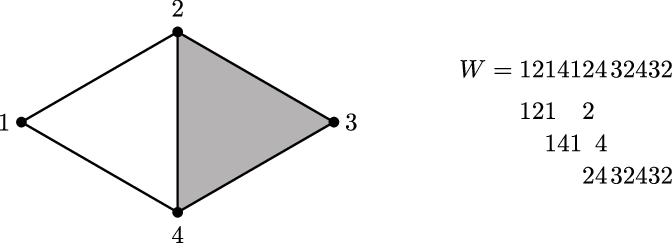}
    \caption{A simplicial complex and a word that $2$-colorfully represents it.}
    \label{fig:colorful-word}
\end{figure}

Consider the simplicial complex with facets $12$, $14$, and $234$ on vertex set $\{1,2,3,4\}$. 
This complex is $2$-colorfully represented by the word $W = 121412432432$. 
Figure \ref{fig:colorful-word} shows this complex and subwords of $W$ which correspond to the faces.
In fact, the figure only shows subwords corresponding to facets, but a $d$-colorful word on alphabet $\sigma$ contains $d$-colorful words on every smaller alphabet (see Lemma \ref{lem:monotone}), so it suffices to find the maximal $d$-colorful subwords. 
\end{example}

\begin{remark} We note that $d$-colorfully representable complexes naturally extend the definition of general $d$-word-representable graphs formulated by Oliveros and Torres~\cite{oliveros_torres}.
A graph is \emph{general $d$-word-representable} if one can find a word whose letters are the vertices of the graph, and whose alternating subwords of length $d+2$ correspond to edges of the graph. Hence general $d$-word-representable graphs are exactly the 1-skeleta of $d$-colorfully representable complexes.
In particular, the two notions coincide for triangle-free graphs. \end{remark}

Depending on one's context, weakly $d$-Tverberg complexes may be a more natural tool than $d$-Tverberg complexes.
Weakly $d$-Tverberg complexes are consistent with the usual approach in other Ramsey-type results that seek to classify the emergence of substructures locally instead of global behavior, which is captured by Tverberg complexes in the strong sense.

\section{Background}

Before proving our main theorems, we first introduce some notation and summarize relevant background material. 
The \emph{nerve} of a collection $\C = \{C_1, \ldots, C_n\}$ of convex sets is the simplicial complex 
\[
\nerve(\C) \eqdef \Big\{\sigma\subseteq [n] \, \Big\vert\, \bigcap_{i\in\sigma} C_i \neq \emptyset \Big\}.
\]
We say that a simplicial complex $\Delta\subseteq 2^{[n]}$ is \emph{partition induced} on a finite set $P$ of points in $\R^d$ if there is a partition $\{P_1,\ldots, P_n\}$ of $P$ so that $\Delta$ is the nerve of the collection $\{\conv(P_1), \ldots, \conv(P_n)\}$. 
With this terminology we can rephrase the definition of (weakly) $d$-Tverberg complexes.
Below, a point set $P\subseteq \R^d$ is in \emph{general position} if for every $k\le d+1$, no set of $k$ points from $P$ lie on a common $(k-2)$-dimensional affine subspace.

\begin{definition}[\cite{altered_nerves}]\label{def:dtverberg}
A simplicial complex $\Delta\subseteq 2^{[n]}$ is \emph{$d$-Tverberg} if there is a constant $C$ so that $\Delta$ is partition induced on any point set $P\subseteq \R^d$ in general position with at least $C$ points. 
\end{definition}

\begin{definition}\label{def:weaklydtverberg}
A simplicial complex $\Delta\subseteq 2^{[n]}$ is \emph{weakly $d$-Tverberg} if there is a constant $C$ so that $\Delta$ is partition induced on a subset of any point set $P\subseteq \R^d$ in general position with at least $C$ points. 
\end{definition}

We will often order point sets $P\subseteq \R^d$, and regard them as sequences.
This allows us to tie the combinatorics of $d$-colorful words to the geometry of Tverberg partitions. 
A \emph{Tverberg partition} of a finite point sequence is a collection of disjoint subsequences whose convex hulls share a common point.
Note that we are abusing terminology slightly, since a Tverberg partition is not necessarily a partition of the entirety of our original sequence.
A Tverberg partition is called \emph{minimal} if deleting any point yields a collection of subsequences whose convex hulls no longer share a common point.
The following multipartite generalization of Kirchberger's theorem, first proved by P\'or \cite{por_thesis}, implies that every minimal Tverberg partition with $r$ parts in $\R^d$ consists of at most $(d+1)(r-1)+1$ points in total (see also the work of Arocha, B\'ar\'any, Bracho, Fabila, and Montejano \cite[Cor.~4]{very_colorful}).

\begin{proposition}\label{prop:multipartite}
Let $P_1,\ldots, P_r$ be disjoint subsets of $\R^d$. Then $\bigcap_{i=1}^r \conv(P_i) \neq \emptyset$ if and only if there exist $P_i'\subseteq P_i$ so that $\bigcap_{i=1}^r \conv(P_i')\neq\emptyset$ and $\bigcup_{i=1}^r P_i'$ contains at most $(d+1)(r-1)+1$ points. 
\end{proposition}

Given a collection of disjoint subsequences $P_1,\ldots, P_r$ of a sequence $P$, one can form a word on alphabet $[r]$ by associating the points of $P_i$ with the letter $i$ for every $i\in[r]$, and writing down the sequence of these letters as they appear in $P$ from start to end.
We say that $P_1,\ldots, P_r$ is a \emph{$d$-colorful partition} if this word is $d$-colorful.
In particular, a $d$-colorful partition will consist of exactly $(d+1)(r-1)+1$ points in total. 
Again note the slight abuse of terminology: a colorful partition is only a partition of a subset of the original sequence. 

There are arbitrarily large point sequences in $\R^d$ whose minimal Tverberg partitions are exactly the $d$-colorful partitions.
One example is a collection of points on the moment curve, chosen so that their coordinates increase heavily (see \cite[page 4]{por} for a discussion). 
Another example, provided by work of Bukh, Loh, and Nivasch \cite[Section 4]{bukh_loh_nivasch}, is the diagonal of the stretched grid. 

A further important fact, established by P\'or \cite[Theorem 1.7]{por}, is that for every positive integer~$m$, any sufficiently large point sequence in strong general position in $\R^d$ contains a subsequence of length $m$ whose minimal Tverberg partitions are exactly the colorful ones. 
Strong general position means, informally, that affine hulls of subsets of the point set intersect generically (see Definition~\ref{def:sgp}). 

We wish to apply P\'or's result to point sequences that are only in general position. 
To do this, it suffices to argue that for every $n$, any sufficiently large general position set must contain a subset of $n$ points in strong general position.
The natural approach to proving this is to build up a strong general position subset of a general position set in an iterative, greedy way.
However, one can become ``stuck" when taking this approach: Doignon and Valette~\cite{doignon-valette} showed that in dimension $d = 5$ there exists a finite set in strong general position that is maximal. 
The following section will circumvent this subtlety. 

\subsection{Large general position sets contain large strong general position subsets.}
To start, we recall the formal definition of strong general position.
\begin{definition}\label{def:sgp}
    A finite set $P \subseteq \R^d$ is said to be in \emph{strong general position} if for any $r \ge 1$ and any pairwise disjoint subsets $P_1, \dots, P_r \subseteq S$ we have that
        \begin{equation}\tag{$\star$}\label{eq:sgp}
        \dim \bigcap_{i=1}^r \mathrm{aff}(P_i) = \max\left \{-1,  d-  \sum_{i=1}^r (d-\dim \mathrm{aff}(P_i))\right\}.
    \end{equation}
\end{definition}

\noindent When $\R^d$ is regarded as a subset of $\RP^d$ and affine hulls are replaced by projective hulls in the above expression, we obtain the analogous notion of being \emph{fully independent}~\cite{doignon-valette}. 

Perhaps surprisingly, strong general position and full independence are not comparable notions.
To see this, first observe that the set of vertices of a parallelogram in $\R^2$ is not in strong general position (opposite sides span parallel lines) but this set is fully independent.
On the other hand, Doignon and Valette~\cite{doignon-valette} constructed a 5-dimensional set with 12 points which is in strong general position, and is not fully independent. 
Moreover, they showed that such sets are exactly the maximal sets in strong general position.
That is, a strong general position set is not fully independent if and only if it is inclusion-maximal among strong general position sets. 

We wish to establish that large general position sets contain strong general position subsets of growing size.
It turns out that it is more straightforward to establish the analogous result for sets that are both in strong general position \emph{and} fully independent. 

\begin{theorem}\label{thm:sgp}
For every $n\ge 1$ and $d\ge 1$, there exists $N = N(n,d)\ge 1$ so that every set of $N$ points in general position in $\R^d$ contains a subset of $n$ points that is in strong general position and fully independent. 
\end{theorem}
\begin{proof}
    Fix a set $Q\subseteq \RP^d$ consisting of $d$ points whose projective hull is the hyperplane at infinity. 
    Our starting point is the observation that if $P\subseteq \R^d$ is such that $P\cup Q$ is fully independent, then $P$ is both fully independent and in strong general position. 
    Indeed, $P$ is fully independent because $P\cup Q$ is, and the only way that $P$ can fail to be in strong general position is for $P$ to contain disjoint subsets $P_1,\ldots, P_r$ so that $\bigcap_{i=1}^r \proj(P_i)$ is nonempty and contained in the hyperplane at infinity---if this intersection contains any real points, then the intersection of the corresponding affine hulls would have the same dimension, satisfying (\ref{eq:sgp}). 
    But then observe that $\proj(Q) \cap \bigcap_{i=1}^r \proj(P_i)$ would be equal to $\bigcap_{i=1}^r \proj(P_i)$, when it should in fact have dimension one less, contradicting the fact that $P\cup Q$ is fully independent.

   With this observation in hand, it will suffice to argue that for every $n\ge 1$ and $d\ge 1$, there exists a constant $N = N(n,d)$ so that every set of $N$ points in general position in $\R^d$ contains a subset $P$ of size $n$ so that $P\cup Q$ is fully independent. 
To establish this, we argue that there is a constant $C(n,d)$ so that if $P\subseteq \R^d$ contains $n-1$ points and $P\cup Q$ is fully independent, then any general position set $X\subseteq \R^d$ of size $C(n,d)$ has $p\in X$ so that $P\cup\{p\}\cup Q$ is also fully independent.
    The choice $N(n,d) = \sum_{i=1}^{n} C(i,d)$ is then sufficient to prove the theorem inductively.
    From the first $N(n-1,d) = \sum_{i=1}^{n-1} C(i,d)$ points in general position, we extract a point set $P$ of size $n-1$, and the remaining $C(n,d)$ points allow us to extract one further point so that we have $n$ points in total whose union with $Q$ is fully independent. 

To prove the existence of $C(n,d)$, it suffices to argue that the set of points in $\R^d$ that can be added to $P\cup Q$ to obtain a larger fully independent set are those in the complement of a union of subspaces with positive codimension, where the number of such subspaces depends only on $n$ and~$d$. 
Doignon and Valette~\cite[Lemma 1]{doignon-valette} showed that if $Z\subseteq \RP^d$ is fully independent and a point $p$ avoids all proper subspaces of the form \[
\proj\left(Z_r\cup\bigcap_{i=1}^{r-1}\proj(Z_i)\right)
\]
where $Z_1,\ldots, Z_r$ are disjoint subsets of $Z$, then $Z\cup \{p\}$ is also fully independent. 
Setting $Z = P\cup Q$, we see that the number of subspaces that must be avoided is at most the number of tuples of disjoint subsets of $P\cup Q$, which depends only on $n$ and $d$, establishing the result. 
\end{proof}

The following theorem was proved for point sets in strong general position by P\'or~\cite{por}, and Theorem~\ref{thm:sgp} allows us to state the result for point sets in general position. 

\begin{theorem}[{\cite[Thm.~1.7]{por}}] 
\label{thm:universal}
Given $d,m,r\in\mathbb N$ with $r\ge 2$, there is $N = N(d,m,r)\in\mathbb N$ such that every sequence of length $N$ in $\R^d$ in general position contains a subsequence of length $m$ whose minimal Tverberg partitions are exactly the colorful ones.
\end{theorem}

\section{Characterizing Weakly $d$-Tverberg Complexes}\label{sec:weak}

\begin{lemma}\label{lem:colorful-partition}
Let $P = (p_1, p_2, \ldots, p_m)$ be a sequence in $\R^d$ whose minimal Tverberg partitions are exactly the $d$-colorful ones. 
A simplicial complex $\Delta\subseteq 2^{[n]}$ is partition induced on $P$ if and only if $\Delta$ is $d$-colorfully represented by a word $W$ of length $m$. 
\end{lemma}
\begin{proof}
First suppose that $\Delta$ is partition induced on $P$, say by a partition $\{P_1,\ldots, P_n\}$. 
Let $W$ be the word on $[n]$ whose $i$-th letter is the unique $j$ so that $p_i\in P_j$.
In other words, $W$ is the word obtained by labeling the points in $P$ in sequence according to which part of the partition they belong to. 
We claim that $W$ $d$-colorfully represents $\Delta$. 

To prove this, let $\sigma\subseteq [n]$.
We aim to show that $\sigma$ is a face of $\Delta$ if and only if $W$ contains a $d$-colorful subword on alphabet $\sigma$.
Note that $\sigma$ is a face of $\Delta$ if and only if $\{P_j\mid j\in\sigma\}$ is a Tverberg partition. 
Since the minimal Tverberg partitions of $P$ are exactly the colorful ones, we conclude that $\sigma$ is a face of $\Delta$ if and only if there are $P_j'\subseteq P_j$ so that $\{P_j'\mid j\in\sigma\}$ is a $d$-colorful Tverberg partition. 
The letters corresponding to points in the various $P_j'$ will form a $d$-colorful subword of $W$ on alphabet $\sigma$, and conversely such a $d$-colorful subword allows us to construct appropriate $P_j'$. 
This proves the first half of the lemma. 

For the converse, suppose that $\Delta$ is $d$-colorfully represented by some word $W$ with length $m$. 
Define a partition $\{P_1,\ldots, P_n\}$ of $P$, where $P_j$ consists of the points $p_i$ so that the $i$-th letter of $W$ is $j$. 
A similar analysis shows that $\Delta$ is partition-induced on $P$ by this partition, proving the result.
\end{proof}

\begin{restatetheorem}{\ref{thm:weak-iff-colorful}}
A simplicial complex is weakly $d$-Tverberg if and only if it is $d$-colorfully representable. 
\end{restatetheorem}
\begin{proof}
Let $\Delta$ be a simplicial complex on vertex set $[n]$. 
First suppose that $\Delta$ is weakly $d$-Tverberg. 
Choose a large sequence of points $P$ in $\R^d$ whose minimal Tverberg partitions are the $d$-colorful partitions.
By choosing $P$ large enough, we can guarantee that $\Delta$ is partition-induced on a subsequence of $P$.
A subsequence also has the property that its minimal Tverberg partitions are exactly the $d$-colorful ones, and so Lemma \ref{lem:colorful-partition} implies that $\Delta$ is $d$-colorfully representable.

For the converse, suppose that $\Delta$ is $d$-colorfully representable by a word $W$ of length $m$. 
By Theorem \ref{thm:universal} every sufficiently large point sequence in general position in $\R^d$ contains a subsequence of length $m$ whose Tverberg partitions are exactly the colorful ones.
Lemma \ref{lem:colorful-partition} then implies that $\Delta$ is partition-induced on this subsequence, so $\Delta$ is weakly $d$-Tverberg.
This proves the result. 
\end{proof}

\section{A Graph That Is Not Weakly $d$-Tverberg}\label{sec:graphs}

We begin by introducing some additional notation. 
Given a word $W$ on alphabet $\sigma$, let $\Delta^d(W)$ be the simplicial complex
\[
\Delta^d(W) \eqdef \{\tau\subseteq \sigma \mid \text{$W$ contains a $d$-colorful subword on $\tau$}\}. 
\]
In other words, $\Delta^d(W)$ is the simplicial complex\footnote{It is not a priori obvious that $\Delta^d(W)$ is a simplicial complex, but this follows from Lemma \ref{lem:monotone}, which we explain in our concluding remarks.} that is $d$-colorfully represented by~$W$. 
If $W$ is any word, let $\red(W)$ denote the \emph{reduced word} formed by deleting any consecutive occurrences of the same letter. 
Observe that $\Delta^d(W) = \Delta^d(\red(W))$.
Finally, if $W$ is a word on alphabet~$\sigma$, and $\tau\subseteq\sigma$, let $W(\tau)$ denote the \emph{restriction} of $W$ to~$\tau$, which is obtained by deleting all letters not in~$\tau$. 
Note that $\Delta^d(W(\tau))$ is the induced subcomplex of $\Delta^d(W)$ on vertex set~$\tau$.

Fix $d\ge 1$.
For a positive integer $n$, let $K(n) \eqdef n(n-1)(d+2)^2 + 1$.
Choose a fixed $n$ large enough that $2^n > \binom{K(n)}{(d+2)^2}$, noting that this is possible because the latter quantity is polynomial in~$n$.
From here on, we let $K$ denote $K(n)$ for our fixed choice of~$n$.
Now define a bipartite graph $G_d$ on a vertex set $A\sqcup B$ as follows.
First, $A$ is simply a set of size $n$.
The second part $B$ consists of $2^{n}(2^K+1)$ vertices: for every $\sigma\subseteq A$, the part $B$ contains $2^K+1$ copies of a vertex whose neighborhood is exactly $\sigma$. 

We aim to prove the following result, from which Theorem \ref{thm:graphs} will follow. 

\begin{theorem}\label{thm:G_d}
The graph $G_d$ is not $d'$-colorfully representable for $d'\le d$.
\end{theorem}

The remainder of this section is dedicated to proving the theorem above. 
Since $G_d$ contains an isomorphic copy of $G_{d'}$ as an induced subgraph for every $d'\le d$, it will suffice to prove that $G_d$ is not $d$-colorfully representable. 
For contradiction, assume that $G_d$ is $d$-colorfully representable, and let $W$ be a word that $d$-colorfully represents it.
For the remainder of this section we treat $W$ as fixed.

Note that $W$ is a word whose letters come from $A\sqcup B$, where $A$ and $B$ are the two parts of the vertex set of $G_d$ as defined above. 
Throughout this section we typically refer to elements of $A\sqcup B$ as ``letters" since we are considering instances of them in the word $W$, though they are also vertices of the graph $G_d$. \bigskip

\noindent \textbf{Chunks.} 
We will often examine $W(A)$, the restriction of $W$ to the set of letters in $A$.
A \emph{chunk} in $W(A)$ is a maximal subword consisting of consecutive instances of the same letter.
Note that a chunk could consist of only a single letter. 
We will regard the letters in a chunk interchangeably as letters in $W$ and letters in $W(A)$.  \bigskip

\noindent \textbf{Claim 1:}
$W(A)$ contains no more than $n(n-1)(d+2)$ distinct chunks. 
Hence $\red(W(A))$ has length at most $n(n-1)(d+2)$. \smallskip

\noindent \textbf{Proof of Claim 1:}
If $W(A)$ has more than $n(n-1)(d+2)$ chunks, then $\red(W(A))$ has length larger than $n(n-1)(d+2)$. 
Then some letter $a$ appears more than $(n-1)(d+2)$ times in $\red(W(A))$. 
Since consecutive letters in $\red(W(A))$ are distinct, there are at least $(n-1)(d+2)$ gaps between the instances of $a$, filled by the remaining $n-1$ letters. 
But then by pigeonhole principle some remaining letter $a'\neq a$ appears in at least $d+2$ distinct gaps. 
Thus $\red(W(A))$ (and hence $W(A)$, and hence $W$) contains a $d$-colorful subword on alphabet $\{a,a'\}$. Hence $aa'$ is an edge in $G_d$, contradicting the fact that $A$ is an independent set in $G_d$.  \bigskip

\noindent \textbf{Insertion patterns.} 
Let $b\in B$. 
Observe that $W(A\cup\{b\})$ is a $d$-colorful representation of the induced subgraph of $G_d$ on vertex set $A\cup\{b\}$.
Let $I_b$ be a word obtained by deleting as many instances of $b$ from $W(A\cup\{b\})$ as possible without changing the graph that it $d$-colorfully represents.
There could be many possibilities for $I_b$, but we will fix one from here on.
We call our fixed $I_b$ the \emph{insertion pattern} of $b$. 

Note that $I_b$ determines the neighborhood of the vertex $b$ in $G_d$.
Also note that when forming $I_b$ we never delete letters from $A$, and consequently we can regard the chunks of $W(A)$ canonically as subwords of $I_b$. 
Hence we will speak below of ``chunks in $I_b$."
Finally, observe that at most one instance of $b$ occurs between two consecutive chunks in $I_b$---any further occurrences would be redundant. \bigskip

\noindent \textbf{Chunk containment.}
An instance of $b$ that occurs in $I_b$ is \emph{contained in} a chunk if there are letters from that chunk both to its left and to its right.
Note that each instance of $b$ that occurs in $I_b$ is contained in at most one chunk, and could be contained in no chunks if it lies between two adjacent chunks, or at the beginning or end of $I_b$. \bigskip

\noindent \textbf{Claim 2:} 
Let $\mathbb{W}$ be the word obtained from $W(A)$ by shortening every chunk to have length at most $d+2$.
Then for every $b\in B$, $\red(I_b)$ can be obtained by inserting non-consecutive copies of $b$ into $\mathbb{W}$, and reducing the resulting word.
\smallskip

\noindent \textbf{Proof of Claim 2:}
We first claim that for any $b\in B$, no chunk in $I_b$ contains more than $d+1$ copies of $b$. 
Suppose for contradiction that some chunk consisting of copies of a letter $a\in A$ contains more than $d+1$ copies of $b$.
None of these copies of $b$ are consecutive (otherwise we could delete one without changing $\Delta^d(I_b)$) and hence this chunk contains an alternating subword on $\{a, b\}$ of length at least $2d+3 > d+2$. 
We now observe that we could delete a copy of $b$ from this chunk without changing the complex $\Delta^d(I_b)$. 
Indeed, the remaining copies of $b$ in the chunk would still form an alternating subword of length $d+2$ with $a$, and any alternating subwords involving $b$ and a letter $a'\neq a$ can use at most one copy of $b$ from this chunk, and so deleting one of the copies of $b$ does not disrupt these alternating subwords either. 
There are no $d$-colorful subwords on alphabets of size three or more in $I_b$, since $\Delta^d(I_b)$ is a graph. 
We chose $I_b$ so that deleting any copy of $b$ would change $\Delta^d(I_b)$, and so we have arrived at a contradiction.

Since no chunk in $I_b$ contains more than $d+1$ copies of $b$, at most $d+2$ letters from each chunk are preserved in $\red(I_b)$.
In particular, $\red(I_b)$ restricted to $A$ is a subword of $\mathbb{W}$, and $\red(I_b)$ can thus be obtained from $\mathbb{W}$ by inserting copies of $b$ and then reducing.
Since $\red(I_b)$ is reduced, we need not insert any consecutive copies of $b$. 
Hence the claim follows.
\bigskip

\noindent \textbf{Claim 3:} 
There exists $\sigma\subseteq A$ so that for every $b\in B$ with neighborhood $\sigma$, $\red(I_b)$ contains at least $(d+2)^2$ instances of the letter $b$.\smallskip

\noindent \textbf{Proof of Claim 3:}
Suppose for contradiction that every $\sigma\subseteq A$ admits some $b\in B$ with neighborhood $\sigma$ and $\red(I_b)$ containing fewer than $(d+2)$ copies of $b$.
Now, consider the process described in Claim 2.
The word $\mathbb{W}$ has length at most $n(n-1)(d+2)^2$ by Claim 1.
Hence each instance of $b$ has at most $K = n(n-1)(d+2)^2 + 1$ places it can be inserted, and by Claim 2 the choices of where to insert $b$ determine $\red(I_b)$.
With at most $(d+2)^2$ instances of $b$ to insert, we obtain at most $\binom{K}{(d+2)^2}$ distinct possibilities for $\red(I_b)$, even as $\sigma$ varies over the $2^n> \binom{K}{(d+2)^2}$ possible subsets of $A$. 
This contradicts the fact that $\sigma$ can be recovered from $\red(I_b)$, and the claim follows.
\bigskip

\noindent \textbf{Claim 4:}
For any $\sigma\subseteq A$, there exist distinct letters $b,b'\in B$ whose neighborhood is $\sigma$, and which have $\red(I_{b})$ and $\red(I_{b'})$ identical, up to replacing $b$ by $b'$.\smallskip

\noindent \textbf{Proof of Claim 4:}
We again use Claim 2. 
There are at most $2^K$ ways to insert copies of a letter $b$ into $\mathbb{W}$, but we constructed $G_d$ so that there are more than $2^K$ vertices with neighborhood exactly $\sigma$. 
Hence at least two of them have the same reduced insertion patterns.\bigskip

\noindent \textbf{Claim 5:} 
Let $\sigma\subseteq A$ be as in Claim 3, and let $b$ and $b'$ be as in Claim 4 with this choice of $\sigma$.
Then $W$ contains a $d$-colorful subword on the alphabet $\{b,b'\}$. \smallskip

\noindent \textbf{Proof of Claim 5:}
We have two relevant facts. 
First, $\red(I_{b})$ and $\red(I_{b'})$ are the same up to replacing $b$ by $b'$. 
Second, each of these reduced insertion patterns contains at least $(d+2)^2$ instances of $b$ and $b'$ respectively.
By Claim 2, the longest consecutive sequence in $\red(I_{b})$ which consists of $b$ alternating with some letter $a\in A$ has length at most $d+2$. 
For each such sequence, let us delete all alternations except for one, obtaining a new word $W_b$. 
Note that $W_b$ still has at least $d+2$ copies of $b$ in it, and instances of letters from $A$ in $W_b$ are in one-to-one correspondence with chunks in $W(A)$. 

Let $W_{b'}$ be obtained by performing the corresponding deletions in $\red(I_{b'})$.
Each instance of $b$ in $W_b$ (respectively, $b'$ in $W_{b'}$) can be associated to an instance of $b$ (respectively, $b'$) in our original word~$W$. 
Let $W'$ be the subword of $W$ on these instances of $b$ and~$b'$. 
Note that $W'$ has length at least $2(d+2)$, and does not contain three consecutive instances of $b$ (respectively $b'$) since each instance of $b$ is paired with an instance of $b'$ which appears within or between the same chunks in~$W$. 
Hence $\red(W')$ has length at least $d+2$, and is a $d$-colorful subword of $W$ on alphabet $\{b,b'\}$. 
This proves the claim.  \bigskip

\noindent \textbf{Summary.} 
We argued that $\red(W(A))$ has bounded length, and used this fact to show that there must be some $\sigma\subseteq A$ so that every letter in $B$ with neighborhood $\sigma$ appears a large number of times, even when we forget the other letters in $B$ and delete duplicate copies. 
Since there are many letters in $B$ with neighborhood $\sigma$, we were able to argue that two of them must appear with the same patterns.
In Claim 5 we showed that these identical patterns imply a $d$-colorful subword of $W$ on these two letters.
But this is a contradiction, since there is not an edge between these two letters in $G_d$.
Hence $G_d$ is not $d$-colorfully representable.
This proves Theorem \ref{thm:G_d}, and hence Theorem \ref{thm:graphs}. 

\section{Concluding Remarks and Questions}

We conclude by establishing some basic facts about $d$-colorfully representable complexes and (weakly) $d$-Tverberg complexes, and raising several questions. 
First, we show that any $d$-colorful word on alphabet $\sigma$ contains a $d$-colorful subword on every alphabet $\tau\subseteq \sigma$.
This implies that every word $d$-colorfully represents \emph{some} simplicial complex; that is, given a word $W$ on alphabet $[n]$, the set of $\sigma$ so that $W$ contains a $d$-colorful subword on $\sigma$ actually does form a simplicial complex, as we would hope.
One could prove this fact geometrically by way of Lemma \ref{lem:colorful-partition}---a colorful word on $\sigma$ would give a Tverberg partition of a point set whose minimal Tverberg partitions are colorful, and deleting a part would yield Tverberg partition with fewer parts, then the lemma applied in converse would yield a colorful subword on a smaller alphabet. 
Below, we give a purely combinatorial proof. 

\begin{lemma}\label{lem:monotone}
If $W$ is a $d$-colorful word on alphabet $\sigma$, then $W$ contains a $d$-colorful subword on $\tau$ for any $\tau\subseteq \sigma$.
\end{lemma}
\begin{proof}
It suffices to prove the result when $\tau = \sigma\setminus \{i\}$ for some $i\in\sigma$. 
Consider the $(d+1)$-many blocks of length $r$ in $W$, which overlap at $d$-many indices. 
Whenever $i$ occurs at an overlap index, let $j$ be the letter immediately preceding it. 
We delete $i$ and the occurrence of $j$ in the block starting at the overlap index. 
Further, we delete all other occurrences of $i$ (these did not appear at overlap indices in the original word).

We have deleted $d+1$ letters in total (one for each block that contains letter $i$ in its ``interior", and two for every pair of blocks which overlap with letter $i$) so we have obtained a word $W'$ of length $(d+1)(r-2) + 1$.
We claim that $W'$ is a $d$-colorful word on alphabet $\sigma\setminus \{i\}$.
Consider any two consecutive blocks in the original word.
After deletions, these become consecutive blocks in~$W'$. 
We consider several cases.

If the original blocks overlapped with letter $i$, then the new blocks overlap at the letter immediately preceding $i$ in the original word, and our choice to delete the second occurrence of this letter guarantees that the new blocks contain each letter from $\sigma\setminus\{i\}$ exactly once. 
Otherwise, the original blocks did not overlap at letter $i$. 
In this case, the second block is altered only by deleting the letter $i$, and so contains every letter from $\sigma\setminus\{i\}$ exactly once.
The first block is altered by deleting letter $i$, and possibly another letter $j$, if $i$ is the first letter in the block. 
If we delete the letter $j$, our choice of $j$ guarantees that it now appears at the first index in the new block. 
In either case, the first block contains every letter from $\sigma\setminus\{i\}$ and the result follows. 
\end{proof}

Lemma \ref{lem:monotone} lets us quickly prove that every complex is $d$-colorfully representable for $d$ large enough.
When combined with Theorem \ref{thm:weak-iff-colorful} we obtain the corollary that every complex is weakly $d$-Tverberg for $d$ large enough.
The proof below is analogous to Oliveros and Torres' proof that every graph is general $d$-word-representable for $d$ large enough \cite[Proposition 2]{oliveros_torres}.

\begin{proposition}\label{prop:m+1}
Every simplicial complex $\Delta$ with $m$ facets is $(m+1)$-colorfully representable.
\end{proposition}
\begin{proof}
Let $\sigma_1, \ldots, \sigma_m$ be the facets of $\Delta$, and for each $i\in[m]$ let $W_i$ be an $(m+1)$-colorful word on alphabet $\sigma_i$. 
Consider the word $W \eqdef W_1W_2\cdots W_m$ obtained by concatenating these colorful words.
We claim that $\Delta$ is $(m+1)$-colorfully represented by $W$.

Clearly $W$ contains an $(m+1)$-colorful word on alphabet $\sigma$ for every $\sigma\in \Delta$.
For the reverse containment, let $F$ be any non-face of $\Delta$. 
Suppose for contradiction that $F$ is a face of $\Delta^{m+1}(W)$.
Then $W$ contains an $(m+1)$-colorful subword $W_F$ on alphabet $F$.
The word $W_F$ has $m+1$ blocks of length $|F|$, each overlapping at their endpoints and containing the letters of $F$ exactly once.
Since $F$ is not a face of $\Delta$, none of the blocks in $W_F$ can consist of letters from a single $W_i$. 
Thus the rightmost letter of the first block in $W_F$ must appear after $W_1$ in $W$, and more generally the rightmost letter of the $i$-th block in $W_F$ must appear after $W_i$ in $W$. 
But $W$ is comprised of only $m$-many $W_i$'s, while $W_F$ has $m+1$ blocks.
Thus we have arrived at a contradiction, and we conclude that $\Delta$ is $d$-colorfully representable as desired. 
\end{proof}

\begin{corollary}
Every simplicial complex with $m$ facets is weakly $(m+1)$-Tverberg.
\end{corollary}

Recall that a cone is a simplicial complex whose facets all share at least one common vertex.
We note that the classes of $d$-Tverberg complexes and weakly $d$-Tverberg complexes coincide for cones. 
Combined with the corollary above, we obtain that every cone is $d$-Tverberg for $d$ large enough.

\begin{proposition}\label{prop:cone}
Let $\Delta$ be a simplicial complex that is a cone.
Then $\Delta$ is $d$-Tverberg if and only if $\Delta$ is weakly $d$-Tverberg.
\end{proposition}
\begin{proof}
Every $d$-Tverberg complex is also weakly $d$-Tverberg, so we only need to argue that if $\Delta$ is weakly $d$-Tverberg then it is $d$-Tverberg.
Let $P$ be a point set in $\R^d$ which is large enough that $\Delta$ is partition induced on some subset $S$ of $P$.
By adding the points of $P\setminus S$ to the part corresponding to the cone vertex in $\Delta$, we obtain a partition of the entire point set $P$ inducing the same nerve. 
\end{proof}

\begin{corollary}
Every cone with $m$ facets is $(m+1)$-Tverberg. 
\end{corollary}

We conclude by explaining that being weakly $d$-Tverberg is a monotone property in the parameter $d$.
We are grateful to Attila P\'or for pointing out the proof of this result, which we originally posed as a question. 

\begin{proposition}\label{prop:monotone}
If $\Delta\subseteq 2^{[n]}$ is $d$-colorfully representable, then it is also $(d+1)$-colorfully representable.
In particular, if $\Delta$ is weakly $d$-Tverberg then $\Delta$ is also weakly $(d+1)$-Tverberg.
\end{proposition}
\begin{proof}
Let $W$ be a word on $[n]$ which $d$-colorfully represents $\Delta$.
By applying a permutation, we may assume that if $i < j$, then the first instance of $i$ in $W$ appears before the first instance of~$j$.
Then consider the word $W'$ which is obtained from $W$ by concatenating all letters from $[n]$ in reverse order to the beginning of $W$.
That is, $W' \eqdef n(n-1)(n-2)\cdots 321 W$.

Observe that if $W$ contains a $d$-colorful word on alphabet $\sigma$, then $W'$ contains a $(d+1)$-colorful word on $\sigma$ as well, since we may use our newly concatenated letters to add one more block to any $d$-colorful subword of $W$. 
Moreover, we claim that the converse holds.
If $W''$ is a $(d+1)$-colorful subword of $W'$ on alphabet $\sigma$, then the only letters of $W''$ that are not from $W$ appear in the first block of $W''$. 
If the letter where the first and second blocks of $W''$ overlap comes from $W$, then we immediately obtain a $d$-colorful subword of $W$ on alphabet $\sigma$.
Otherwise, the letter where this overlap occurs is equal to the minimum of $\sigma$.
But then an instance of this letter occurs in $W$ before all instances of other letters in $\sigma$.
Hence we also obtain a $d$-colorful subword of $W$ on $\sigma$ in this case. 
This proves the result. 
\end{proof}

\noindent Our results above motivate several further questions, which we pose below.
For the third question, note that it is at least a decidable problem to determine whether or not a simplicial complex is weakly $d$-Tverberg, since a word that $d$-colorfully represents it requires a bounded number of letters to represent each face. 
However, it is not known if one can algorithmically decide whether or not a simplicial complex is $d$-Tverberg. 
If recognizing $d$-Tverberg complexes is an undecidable problem, this would additionally provide a negative answer to our second question below.

Our construction of bipartite graphs that are not $d$-Tverberg used a very large number of vertices. 
It is likely that there are much smaller bipartite graphs that are not $d$-Tverberg, and for $d=3$ Oliveros and Torres~\cite{oliveros_torres} proposed a candidate on ten vertices, motivating our fourth question.

\begin{enumerate}
    \item Is every $d$-Tverberg complex also $(d+1)$-Tverberg?
    \item Do the classes of $d$-Tverberg and weakly $d$-Tverberg complexes coincide?
    \item For a fixed $d\ge 2$, what is the computational complexity of deciding whether or not a given simplicial complex is (weakly) $d$-Tverberg?
    \item How many vertices are needed to construct a bipartite graph that is not $3$-Tverberg?
\end{enumerate}

\section*{Acknowledgements}

We are grateful to Jes\'us De Loera, Deborah Oliveros, and Attila P\'or for helpful discussions. 
We also extend our thanks to the anonymous referees for their feedback.

\bibliographystyle{plain}
\bibliography{d-tverberg.bib}

\end{document}